\documentclass[12pt]{amsart}
\usepackage{amssymb,amsmath,latexsym,enumitem,soul,fancyhdr,titletoc,lipsum,units,graphicx}
\usepackage{mathrsfs}
\usepackage[colorlinks=true,linkcolor=blue,citecolor=red]{hyperref}
\theoremstyle{plain}
\newtheorem{theorem}{Theorem}[section]
\newtheorem{corollary}[theorem]{Corollary}
\newtheorem{lemma}[theorem]{Lemma}

\newtheorem*{theorem*}{Theorem}
\newtheorem*{corollary*}{Corollary}

\theoremstyle{definition}

\newtheorem{remark}{Remark}[section]
\newtheorem*{remark*}{Remark}

\newcommand{\rk}{\operatorname{rank}}

\newcommand{\rad}{\operatorname{rad}}

\newcommand{\ord}{\operatorname{ord}}

\newcommand{\Mod}[1]{(\operatorname{mod}\ #1)}

\numberwithin{equation}{section}

\begin{document}

%\title[On the fixed points of $x \mapsto x^x$ modulo a prime]{On the
%  fixed points of $x \mapsto x^x$ modulo a prime} 
\title[On the fixed points of the map $x \mapsto x^x$ modulo a prime, II]{On the fixed points of the map $x \mapsto x^x$ modulo a prime, II} 
\author[Adam Tyler
Felix]{Adam Tyler Felix} \address{Department of Mathematics, KTH,
  Royal Institute of Technology, 100 44 Stockholm, Sweden}
\email{atfelix@kth.se} \author[P\"{a}r Kurlberg]{P\"{a}r Kurlberg}
\address{Department of Mathematics, KTH, Royal Institute of
  Technology, 100 44 Stockholm, Sweden} \email{kurlberg@kth.se}
\thanks{A.F. supported by the G\"{o}ran Gustafsson Foundation for
  Research in Natural Sciences and Medicine. P.K.  was partially
  supported by grants from the G\"oran Gustafsson Foundation for
  Research in Natural Sciences and Medicine, and the Swedish Research
  Council (621-2011-5498).  } 

\begin{abstract}
  We study number theoretic properties of the map
  $x \mapsto x^{x} \Mod{p}$, where $x \in \{1,2,\ldots,p-1\}$, and
improve on some recent upper bounds, due to Kurlberg, Luca, and
Shparlinski, on the number of primes $p < N$ for which the map only
has the trivial fixed point $x=1$.
A key technical result, possibly of independent interest, is the
existence of  subsets $\mathscr{N}_{q} \subset
\{2,3,\ldots,q-1\}$ such that
almost all $k$-tuples of distinct integers $n_{1}, n_{2},\ldots,n_{k} \in
\mathscr{N}_q$ are multiplicatively independent (if $k$ is not
too large), and $|\mathscr{N}_q| = q \cdot (1+o(1))$ as $q \to \infty$.
For $q$ a large prime, this is  used to show that the number of solutions
to a certain large and sparse system of $\mathbb{F}_q$-linear forms $\{
\mathscr{L}_{n} \}_{n=2}^{q-1}$ ``behaves randomly'' in the sense that
$|\{ \mathbf{v} \in \mathbb{F}_{q}^{d} :  \mathscr{L}_{n}(\mathbf{v})
=1, n = 2,3, \ldots, q-1 \}| \sim q^{d}(1-1/q)^{q} \sim q^{d}/e$.
(Here $d=\pi(q-1)$ and the coefficents of $\mathscr{L}_{n}$ are given
by the exponents in the prime power factorization of $n$.)

\end{abstract}
%\date{\today}
%\date{March 6, 2016}
%\date{July 13, 2016}
\date{June 26, 2017}
%\timestamp
%\fixme{Remove timestamp}
\maketitle

%\vspace{-.75cm}
%\tableofcontents
%\titlecontents{subsection}[3.8em]{\contentslabel{2.3em}}{\hspace*{-2.3em}}{\titlerule*[1pc]{.}\contentspage}

\section{Introduction}\label{sec:1}

%\subsection{Motivation and Main Result}\label{subsec:1.1}

For a prime $p$, let
$\psi_{p}: \{1,2,\dotsc,p-1\} \to \{1,2,\dotsc,p-1\}$ be the remainder
of $x^{x}$ divided by $p$.  The function $\psi_{p}$ has
cryptographic applications related to variations of the ElGamal
signature scheme (see \cite[Notes 11.70 and 11.71]{MvOV}); our
main focus is studying the number of non-trivial fixed points of
$\psi_p$ as $p$ varies.
Let $$F(p) := \#\{x \in \{1,2,\dotsc,p-1\}: \psi_{p}(x) = x\}$$ denote the
number of fixed points of $\psi_{p}$.  
For convenience, we will slightly abuse
notation and simply write $\psi_{p}(x) = x^{x} \Mod{p}$ (note that
$x^{x}$ is not well defined modulo $p$.)  As $1$ is always a
fixed point of $\psi_p$ we will say it is trivial; all other fixed
points are said to be {\em nontrivial}.
%Since $1$ is always a fixed point,
%we will  focus on non-trivial fixed points, which correspond to $x
%\in \{2,3,\dotsc,p-1\}$ such that $x^{x-1} \equiv 1 \Mod{p}$.  

Kurlberg, Luca and Shparlinski \cite{KLS} gave bounds on the number of
primes $p$ for which $\psi_p$ only has trivial fixed points.
More
specifically, they show most primes $p$ have at least one fixed point
besides $1$: with $\mathcal{A}(N) = \{p \le N: F(p) = 1\}$ they proved
that (cf. \cite[Theorem 1]{KLS})
\begin{equation}\label{eq:1.2}
\#\mathcal{A}(N) \le \frac{\pi(N)}{(\log_{3}N)^{\vartheta+o(1)}}
\end{equation}
as $N \to \infty$, where $\pi(x) := \#\{p \le x: p \text{ is prime}\}$
is the prime counting function and
\begin{equation}\label{eq:1.3}
\vartheta=\frac{1}{\zeta(2)}-\frac{1}{2\zeta(2)^{2}}=\frac{6\pi^{2}-18}{\pi^{4}}
\approx 0.4231394212\dotsm, 
\end{equation}
$\log x := \max\{\ln x, 2\}$, and $\log_{k} := \log(\log_{k-1} x)$
for $k \in \mathbb{N}$ and $k \ge 2$.

In \eqref{eq:1.2}, the exponent $\vartheta$ is 
related to the number of solutions to a certain system of 
linear forms modulo $q$, where $q$ is a prime.  
For the convenience of the reader, we briefly describe how solutions
to linear forms modulo $q$ are related to fixed points of $\psi_p$
(cf. 
\cite[Section~2]{KLS} for more details): For
primes $p \equiv 1 \mod{q}$, it turns out that $\psi_p$ has a
nontrivial fixed point if $n/q$ is a $q$-th power modulo $p$, for some
integer $n \in [1, q-1]$.  This in turn can be characterised in terms
of the image of Frobenius, acting on
$\operatorname{Gal}\left(\mathbb{Q}(\sqrt[q]{1}, \sqrt[q]{2}, \ldots,
  \sqrt[q]{q-1}, e^{2\pi i/q})/\mathbb{Q}(e^{2\pi i/q})\right)$,
lying in a certain union of conjugacy classes.  The cardinality of
said union is related to the number of solutions, modulo $q$, to the
following system of linear equations.
%
%In \eqref{eq:1.2}, the exponent $\vartheta$ was determined by
%considering the 
%following system of linear forms modulo $q$, where $q$ is a prime: let
%Let $d=\pi(q-1)$, where $\pi(x)$ is the prime counting function.  For
Let $d=\pi(q-1)$, and for
$1 \le n \le q-1$, let
\begin{equation*}
n = \prod_{i=1}^{d} p_{i}^{\mu_{i}(n)}
\end{equation*}
be the prime power factorization of $n$, where we have ordered the
primes  $p \le q-1$ so that
$p_{1} < p_{2} < \dotsb < p_{d} < q$.  For $n \in \mathbb{Z} \cap [1,
q-1]$, define linear forms
$\mathscr{L}_{n}:\mathbb{F}_{q}^{d} \to \mathbb{F}_{q}$ by 
\begin{equation}\label{eq:1.4}
\mathscr{L}_{n}(\mathbf{v}) := \sum_{i=1}^{d} \mu_{i}(n)v_{i},
\end{equation}
where $\mathbf{v} := (v_{1},v_{2},\dotsc,v_{d}) \in \mathbb{F}_{q}^{d}$. 
%
%For $x_{0} \in \mathbb{F}_{q} \setminus \{0\}$ and 
%$S \subset \{1,2,\dotsc,q-1\}$, 
For $x_{0} \in \mathbb{F}_{q}^{\times}$ fixed, let
\begin{equation*}
N_{q} := N_{q}(x_{0}) = \#\Bigl\{\mathbf{v} \in \mathbb{F}_{q}^{d}: \mathscr{L}_{n}(\mathbf{v}) \ne x_{0} \text{ for all } n \in \{1,2,3,\dotsc,q-1\}\Bigr\}
%M_{q,S}(x_{0}) &:= \#\Bigl\{\mathbf{v} \in \mathbb{F}_{q}^{d}: \mathscr{L}_{n}(\mathbf{v}) = x_{0} \text{ for all } n \in S\Bigr\}.
\end{equation*}
and put $c(q) := N_{q}/q^{d}$.  Kurlberg, Luca and Shparlinski showed
that 
$\#\mathcal{A}(N) \le \frac{\pi(N)}{(\log_{3}N)^{1-c(q) + o(1)}}$,
gave the  bound  (cf. \cite[Lemma 3]{KLS}) 
\begin{equation}
\label{eq:cq-bound}
c(q) \le 1-\vartheta+o(1)
\end{equation}
and conjectured\footnote{The conjecture was mistakenly stated for any
  $x_{0} \in \mathbb{F}_{q}$, but it is essential to assume that
  $x_{0} \neq 0$ since the form $\mathscr{L}_{1}$ is the zero form,
  and hence $\mathscr{L}_{1}(\mathbf{v}) = 0$ for all
  $\mathbf{v} \in \mathbb{F}_{q}^{d}$.  The upper bound
  \eqref{eq:cq-bound} is  valid without any assumption on
  $x_{0}$, as it is based on examining square-free values of $n \geq 2$.} that 
$c(q) = e^{-1}+o(1)$.  The basis for the conjecture is the following
probabilistic heuristic: if $q$ is large, $\mathbf{v} \neq 0$, and the
linear forms $\{ \mathscr{L}_{n} \}_{n=2}^{q-1}$ are {\em random},
then the probability that $\mathscr{L}_{n}(\mathbf{v} ) \neq x_{0}$
for all $n$ equals $(1-1/q)^{q-2} = 1/e + o(1)$.  Summing over all
nonzero $\mathbf{v}$ and using the linearity of expectations, we find
that the expected value of $N_{q}(x)$ is $q^{d} \cdot (1/e + o(1))$.

Of course the collection of linear forms is far from random, e.g., the
number of nonzero coefficients of $\mathscr{L}_{n}$ equals $\omega(n)$
(the number of distinct prime divisors of $n$); for $n < q$ we find
that $\omega(n) \leq \log q = d^{o(1)}$ and hence
$\{ \mathscr{L}_{n} \}_{n=2}^{q-1}$ is a collection of quite sparse
linear forms (in the sense that most coefficients are zero).
Moreover, as $\mu_i(n) \leq 1$ if $p_{i} > \sqrt{q}$, most
coefficients of the linear forms are very small.
Nonetheless, the above heuristic turns out to give the
correct answer.
%, and 
%the main result of this paper is that the conjecture is true.
\begin{theorem}\label{thm:1.1}
As $q \to \infty$,
\begin{equation}
\label{eq:main-theorem}
%c(q) =
%\frac{1}{e}+O\left(\frac{(\log_{3}q)^{2}}{(\log_{2}q)(\log_{4}q)^{2}}\right).
c(q) = \frac{1}{e}+O\left(\frac{1}{\log_{2}q}\right).
\end{equation}
\end{theorem}
\begin{remark}
  The method of proof would give a similar result in (roughly) the
  following setting.  Assume that $L_{q}$
  is a finite collection of non-zero distinct linear forms modulo $q$
  having the properties that (1): there exists a subset
  $L_{q}' \subset L_{q}$
  such that $|L_{q}'| = (1+o(1/q)) |L_{q}| = (1+o(1))q$.
  (2): for almost all $k$-tuples
  $L_{k,q}$
  of distinct forms in $L_{q}'$,
  the forms in $L_{k,q}$
  are linearly independent, for $2 \leq k \leq K_{q}$,
  where $K_{q}$
  (slowly) tends to infinity with $q$.
  (3): The number of $k$
  tuples of distinct forms $L_{k,q}$
  whose rank $r \leq k-1$ is $|L_{q}|^{r-o(1)}$.
\end{remark}

We have the following corollary of Theorem \ref{thm:1.1} and \cite[pp. 154-155]{KLS}:
\begin{corollary}\label{cor:1.2}
As $N \to \infty$,
\begin{equation*}
\#\mathcal{A}(N) \le \frac{\pi(N)}{(\log_{3}N)^{1-\frac{1}{e} + o(1)}}.
\end{equation*}
\end{corollary}
For comparison with (\ref{eq:1.3}), note  that $1-\frac{1}{e} \approx
%For comparison with (\ref{eq:1.3}), note  that $1-1/e \approx
0.63212\dotsm$.
Also, if one wishes to be explicit, then $o(1)$ in the exponent
becomes
%\begin{equation*}
$
%O\left(\frac{(\log_{4} N)(\log_{6} N)^{2}}{(\log_{5} N)(\log_{7}
%N)^{2}}\right).
O\left(\frac{\log_{5} N}{\log_{4} N}\right).
$
%\end{equation*}
For more details, see \cite[\S2]{KLS}.

%\newpage
\subsection{Outline of the proof}\label{subsec:1.3}

Since $\mathscr{L}_{1}$ is the zero form and $x_{0} \neq 0$, it is
enough to consider $\mathbf{v} \in \mathbb{F}_{q}^{d}$ such that
$\mathscr{L}_{n}(\mathbf{v}) \ne x_{0} \text{ for all } n \in
\{2,\dotsc,q-1\}$. In \S\ref{sec:2}, we then reduce the problem of 
%
% counting
% $N_{q} :=N_{q}(x_{0})$ for $x_{0} \in \mathbb{F}_{q}^{\times}$.  
%
determining $N_{q}(x_{0})$ for $x_{0} \in \mathbb{F}_{q}^{\times}$ to
that of finding $N_{q} :=N_{q}(1)$.
We further note that for any subset $\mathscr{N} \subset
\{2,3,4,\dotsc,q-1\}$,  
\begin{align*}
N_{q} &= \#\{\mathbf{v} \in \mathbb{F}_{q}^{d}: \mathscr{L}_{n}(\mathbf{v}) \ne 1 \text{ for all } n \in \{2,\dotsc,q-1\}\} \\
&\le \#\{\mathbf{v} \in \mathbb{F}_{q}^{d}:  \mathscr{L}_{n}(\mathbf{v}) \ne 1 \text{ for all } n \in \mathscr{N}\} \\
&= 
M_{q,\mathscr{N}}  =
\sum_{k=0}^{N} (-1)^{k}\sum_{\substack{S \subset \mathscr{N} \\ |S| = k}} M_{q,S},
\end{align*}
where
\begin{equation*}
M_{q,S} := \#\{\mathbf{v} \in \mathbb{F}_{q}^{d}: \mathscr{L}_{n}(\mathbf{v}) = 1 \text{ for all } n \in S\}.
\end{equation*}
%Moreover, the  Bonferroni inequalities implies that for any $K \in
%\mathbb{N}$,
In particular, truncating the inclusion/exclusion at an odd, or even,
number of terms gives the following bounds on
$M_{q,\mathscr{N}}$, 
for any $K \in \mathbb{N}$:
\begin{equation*}
\sum_{k=0}^{2K-1} (-1)^{k}\sum_{\substack{S \subset \mathscr{N} \\ |S|
    = k}} M_{q,S} 
\le 
% N_{q} 
% \#\{\mathbf{v} \in \mathbb{F}_{q}^{d}:  \mathscr{L}_{n}(\mathbf{v})
%\ne 1 \text{ for all } n \in \mathscr{N}\} 
M_{q,\mathscr{N}}
%\\
\le \sum_{k=0}^{2K} (-1)^{k}
\sum_{\substack{S \subset \mathscr{N} \\ |S| = k}} M_{q,S}. 
\end{equation*}
(These combinatorial bounds appears in many places in number theory,
e.g. in Brun's pure sieve.)  Let
\begin{equation*}
\Sigma := \Sigma_{K} := \sum_{k=0}^{K} (-1)^{k}\sum_{\substack{S \subset \mathscr{N} \\ |S| = k}} M_{q,S}.
\end{equation*}
Observe that, if $S$ is a set of
$\mathbb{F}_{q}$-independent linear 
forms, then $M_{q,S} = q^{d-|S|}$ and this quickly yields  the
main term.
Estimating the error term is more difficult;
it amounts to determining the contribution from
$M_{q,S}$ as $S$ ranges over sets of $\mathbb{F}_{q}$-dependent forms.
Our strategy is to first reduce the problem of
$\mathbb{F}_{q}$-independence of subsets of forms
$\{ \mathscr{L}_{n} \}_{n=2}^{q-1}$ to multiplicative independence of
subsets of $\{2,3,\ldots,q-1 \}$ (see Lemma \ref{lemma:2.2}).  
A key technical result, perhaps of independent interest, is then that
there exists large subsets $\mathscr{N}_q \subset \{2, 3, \ldots, q-1\}$ such
that essentially all $k$-tuples of distinct elements of
$\mathscr{N}_q$ are multiplicatively independent, provided $k$ is not
too large.  Before stating the result we introduce the following
convenient notation: given a set $\mathscr{A}$ and $k \in \mathbb{N}$,
let
$\mathscr{A}^{[k]} := \{\mathscr{B} \subset \mathscr {A}:
|\mathscr{B}|=k\}$.
\begin{theorem}\label{thm:3.1}
For each integer $q$ there exists $\mathscr{N}_q \subset \{2, 3, \ldots,
q-1\}$ such that, as $q \to \infty$,
$$
\# \mathscr{N}_q = q + O(q/\log_2 q)
$$
(where the implied constant is less than $1$) and 
\begin{multline}
\label{eq:better-estimate}
\#\left\{\mathscr{S} \in \mathscr{N}_{q}^{[k]}: \mathscr{S} \text{ is
    multiplicatively independent}\right\}  
\\ =
\binom{\# \mathscr{N}_q}{k} + O( (\# \mathscr{N}_q)^{k-3/2+o(1)})
 =
\frac{q^{k}}{k!}+O\left(\frac{q^{k}}{(k-1)!\log_{2}q}\right),
\end{multline}
provided that  $k=o(\sqrt{\log_{2}q})$.
%
% In particular,
% \begin{equation*}
% \#\left\{\mathscr{S} \in \mathscr{N}_{q}^{[k]}: \mathscr{S} \text{ is
%     multiplicatively independent}\right\}  =
% \frac{q^{k}}{k!}(1+o(1)). 
% \end{equation*}
\end{theorem}
Using Theorem \ref{thm:3.1} we easily obtain a sufficiently good upper
bound on $N_{q}$.  To obtain a lower bound we remove all
$\mathbf{v} \in \mathbb{F}_{q}^{d}$ such that
$\mathscr{L}_{n}(\mathbf{v}) = 1$ for some $n$ in the complementary
set
$\mathscr{N}_q^{c} = \{2,3, \ldots, q-1\} \setminus \mathscr{N}_{q}$.
As $\# \mathscr{N}_q^{c} = O(q/\log_{2}q)$, a sufficient upper bound
on the number of removed $\mathbf{v}$ follows easily (see
\S\ref{subsec:3.5}.)

\begin{remark}
  For recent results on asymptotics for the number of multiplicatively
  dependent $k$-tuples (not necessarily distinct)
  whose coordinates are algebraic numbers of bounded height, see
  \cite{pappalardi-sha-shparlinski-stewart}.  In particular,
  \cite[Theorem~1.1]{pappalardi-sha-shparlinski-stewart} gives an
%  asymptotic for the error term in Theorem~\ref{thm:3.1}, though not
  asymptotic for the number of multiplicatively dependent $k$-tuples,
  though not   uniform in $k$.
  On the other hand, using \cite[Corollary~3.2]{loher-masser} (due to
  K. Yu) to find ``short'' exponent vectors in multiplicative
  relations leads to a good upper bound with a 
 significant improvement 
%on the error term,
 % with quite good 
in the level of uniformity in $k$. We
  thank Igor Shparlinski for pointing this out.
%
%We also note that 
% Theorem~\ref{thm:3.1} holds for all $q \in \mathbb{Z}^+$.
\end{remark}

% Then,
% we select $\mathscr{N}$ so that $N := \#\mathscr{N} = q + o(q)$ and
% its elements have few large prime factors.  This second condition
% limits the possible choices for values of $S$ (see \S
% \ref{subsec:3.3}).  The upper and lower bounds on $N_{q}$ are
% straightforward given this work.

\subsection{Related results}\label{subsec:discussion}

Little is known about the dynamics and distribution of $\psi_{p}$.
The proof technique for \cite[Theorem 4]{BBS} implies
$F(p) \le p^{\frac{1}{3}+o(1)}$.  In \cite{FrHo}, Friedrichsen and
Holden introduced a probabilistic model for $F(p)$: the distribution
of $F(p)$ should be closely related to $\sum_{d|p-1} X_{d}$, where
$X_{d}$ ranges over independent random variables having binomial
distributions with parameters $(\phi(d), 1/d)$; they also gave 
numerical evidence for the validity of this model.  See
\S\ref{subsec:1.2} for further numerical investigations.
Further, in
\cite[Section~3]{KLS}, a heuristic argument that
$\sum_{p\leq N} F(p) = (1+o(1)) N$ was given.

As for lower and upper bounds on the size of the image, by Crocker
\cite{Crocker} and Somer \cite{Somer},  we know that
\begin{equation*}
\left[\sqrt{\frac{p-1}{2}}\right] \le \#\Bigl\{\psi_{p}(x): x \in \{1,2,\dotsc,p-1\}\Bigr\} \le \frac{3}{4}p+O\left(p^{\frac{1}{2}+o(1)}\right).
\end{equation*}

There are also upper bounds on the cardinality of preimages: with
\begin{equation*}
N(p,a) := \#\Bigl\{x \in \{1,2,\dotsc,p-1\}: \psi_{p}(x) \equiv a
\Mod{p}\Bigr\}, 
\end{equation*}
and
\begin{equation*}
M(p) := \#\Bigl\{(x,y) \in \{1,2,\dotsc,p-1\}^{2}: \psi_{p}(x) =
\psi_{p}(y)\Bigr\}, 
\end{equation*}
Balog, Broughan and Shparlinski \cite[Corollary 5, Theorem 7 and
Theorem 8]{BBS} showed the following uniform bounds for $a$ with
$\gcd(a,p)=1$ and multiplicative order $t$: 
\begin{equation}\label{eq:1.1}
N(p,a) \le \min\left\{p^{\frac{1}{3} + o(1)}t^{\frac{2}{3}}, p^{1 + o(1)}t^{-\frac{1}{12}}\right\}
\end{equation}
and
\begin{equation*}
M(p) \le p^{\frac{48}{25}+o(1)}.
\end{equation*}
\indent Let $a=1$.  Then, as noted in \cite{BBS}, \eqref{eq:1.1}
implies $N(p,1) \le p^{\frac{1}{3}+o(1)}$.  Cilleruelo and Garaev
\cite{CG1, CG2} improve these bounds to $N(p,1) \le
p^{\frac{27}{82}+o(1)}$ and $M(p) \le p^{\frac{23}{12}+o(1)}$.

\subsection*{Acknowledgements} 
We would like to thank Florian Luca and Igor Shparlinski for their
comments on an early version of the paper.  
%We are also grateful for the 
%
We would also like to thank the two anonymous referees for their
careful reading of the paper and for comments that greatly improved
the exposition, as well as leading to a sharper formulation of
Theorem~\ref{thm:3.1}.
 %, and streamlined some of the proofs.

\section{Notation}\label{subsec:1.4}

The letters $p$, $q$ and $\ell$ denote prime numbers.  The letters
$d$, $k$, $m$, $n$, $r$, $s$ and $t$ denote natural numbers.  Letters
of the form $\mathbf{v}$ and $\mathbf{w}$ denote vectors in
$\mathbb{F}_{q}^{d}$.   For $n \in \mathbb{N}$, $\rad(n)$ and
$P(n)$ respectively denote the largest squarefree divisor and the
largest prime divisor of $n$.  
We write $p^{\alpha} \| n$ if $p^{\alpha} \mid
n$ and $p^{\alpha+1} \nmid n$, and the function $\nu_{\ell}(n)$ denotes
the maximum power of $\ell$ that divides $n$.  That is,
$\nu_{\ell}(n)=k$ means $\ell^{k} \| n$.  We say that
$n_{1},n_{2},\dotsc,n_{r}$ are \textbf{multiplicatively independent}
if $\alpha_{1}=\alpha_{2}=\dotsb=\alpha_{r}=0$ is the only integer
solution to $n_{1}^{\alpha_{1}}n_{2}^{\alpha_{2}}\dotsm
n_{r}^{\alpha_{r}}=1$.  Otherwise, $n_{1},n_{2},\dotsc,n_{r}$ are
\textbf{multiplicatively dependent}.  The linear form
$\mathscr{L}_{n}$, where $n \in \mathbb{N}$, is defined in
\eqref{eq:1.4}.  We say
$\mathscr{L}_{n_{1}},\mathscr{L}_{n_{2}},\dotsc,\mathscr{L}_{n_{k}}$
are $\mathbb{F}_{q}$-\textbf{independent} if
$\alpha_{1}=\alpha_{2}=\dotsb=\alpha_{k}=0$ with $\alpha_{i} \in
\mathbb{F}_{q}$ for all $i \in \{1,2,\dotsc,k\}$ is the only solution
to 
\begin{equation*}
\alpha_{1}\mathscr{L}_{n_{1}}(\mathbf{v})+\alpha_{2}\mathscr{L}_{n_{2}}(\mathbf{v})+\dotsb+\alpha_{k}\mathscr{L}_{n_{k}}(\mathbf{v}) = 0(\mathbf{v}) = 0
\end{equation*}
for all $\mathbf{v} \in \mathbb{F}_{q}^{k}$.  Otherwise,
$\mathscr{L}_{n_{1}},\mathscr{L}_{n_{2}},\dotsc,\mathscr{L}_{n_{k}}$
are called $\mathbb{F}_{q}$-\textbf{dependent}. 

%For $x \in \mathbb{R}$, $\pi(x) := \#\{p \le x\}$.  We define
Recall that $\pi(x) := \#\{p \le x\}$, and that we define
$\log_{k}(x)$ for $x \in \mathbb{R}_{> 0}$ and $k \in \mathbb{N}$
iteratively: $\log x = \log_{1} x = \max\{\ln x, 2\}$ and
$\log_{k} x = \log(\log_{k-1} x)$ for $k \in \mathbb{N}$ and
$k \ge 2$.  Let $f: X \to \mathbb{C}$ and $g:X \to \mathbb{R}_{\ge 0}$
be functions.  By the equivalent notations $f(x)=O(g(x))$ or
$f \ll g$, we mean there exists a constant $C$ such $|f(x)| \le Cg(x)$
for all $x \in X$.  The constant $C$ is called the \textbf{implied
  constant} when writing $f(x)=O(g(x))$.  If the implied constant is
dependent on some parameter $P$, then we write $f(x)=O_{P}(g(x))$ or
$f(x) \ll_{P} g(x)$.  We write $f(x) \asymp g(x)$, $f(x) \sim g(x)$
and $f(x)=o(g(x))$ to signify $f(x) \ll g(x) \ll f(x)$,
$f(x)/g(x) \to 1$ and $f(x)/g(x) \to 0$ as $x \to \infty$ with
$x \in X$, respectively.

\section{Lemmata}\label{sec:2}

%Here are some useful lemmata that may be needed to complete the proof of Theorem \ref{thm:1.1}.
We first reduce the problem using the following lemmas.
\begin{lemma}\label{lemma:2.1}
If $x_{0} \in \mathbb{F}_{q}^{\times}$ then $N_{q}(x_{0}) =
N_{q}(1)$. % and $M_{q,S}(x_{0}) = M_{q,S}(1)$.%Also, $N_{q}(0) = 0$. 
\end{lemma}
\begin{proof}
%The last statement holds since $\mathscr{L}_{1}(\mathbf{v}) = 0$ for all $\mathbf{v} \in \mathbb{F}_{q}^{d}$.
The statements follow since $T_{x_{0}}: \mathbb{F}_{q}^{d} \to
\mathbb{F}_{q}^{d}$ defined by $T_{x_{0}}(\mathbf{v}) =
x_{0} \cdot \mathbf{v}$ is an isomorphism if $x_{0} \in \mathbb{F}_{q}^\times
$.
\end{proof}
As such, denote $N_{q} = N_{q}(1)$. % and $M_{q,S} := M_{q,S}(1)$.
\begin{lemma}\label{lemma:2.2}
Let $k \in \mathbb{N}$.
\begin{enumerate}[label=(\alph{enumi})]
\item\label{lemma:2.2(a)} If $n_{1},n_{2},\dotsc,n_{k}$ are
  multiplicatively dependent, then the forms
  $\mathscr{L}_{n_{1}},\mathscr{L}_{n_{2}},\dotsc,\mathscr{L}_{n_{k}}$
  are $\mathbb{F}_{q}$-dependent. 
\item\label{lemma:2.2(b)} Suppose $k < \frac{\log q}{10 \log_{2}q}$.  Then, $n_{1},n_{2},\dotsc,n_{k} \in \{2,3,\dotsc,q-1\}$ are multiplicatively independent if and only if $\mathscr{L}_{n_{1}},\mathscr{L}_{n_{2}},\dotsc,\mathscr{L}_{n_{k}}$ are $\mathbb{F}_{q}$-independent.
\end{enumerate}
\end{lemma}
\begin{proof}
Let $n_{1},n_{2},\dotsc,n_{k} \in \{2,3,\dotsc,q-1\}$ be distinct.  Suppose $n_{i}$ has prime power factorization $n_{i} = p_{1}^{e_{i,1}}p_{2}^{e_{i,2}} \dotsm p_{d}^{e_{i,d}}$, where $e_{i,j}=0$ is permissible.
\begin{enumerate}[label=(\alph{enumi})]
\item Suppose $n_{1},n_{2},\dotsc,n_{k}$ are multiplicatively dependent.  Then, there exist integers $\alpha_{1},\alpha_{2},\dotsc,\alpha_{k}$ such that $n_{1}^{\alpha_{1}}n_{2}^{\alpha_{2}} \dotsm n_{k}^{\alpha_{k}} = 1$.  In particular,
\begin{align*}
1 &= \left(p_{1}^{e_{1,1}}p_{2}^{e_{1,2}} \dotsm p_{d}^{e_{1,d}}\right)^{\alpha_{1}}\left(p_{1}^{e_{2,1}}p_{2}^{e_{2,2}} \dotsm p_{d}^{e_{2,d}}\right)^{\alpha_{2}} \dotsm \left(p_{1}^{e_{k,1}}p_{2}^{e_{k,2}} \dotsm p_{d}^{e_{k,d}}\right)^{\alpha_{k}} \\
&= p_{1}^{\alpha_{1}e_{1,1}+\alpha_{2}e_{2,1}+\dotsb+\alpha_{k}e_{k,1}}p_{2}^{\alpha_{1}e_{1,2}+\alpha_{2}e_{2,2}+\dotsb+\alpha_{k}e_{k,2}} \dotsm p_{d}^{\alpha_{1}e_{1,d}+\alpha_{2}e_{2,d}+\dotsb+\alpha_{k}e_{k,d}}.
\end{align*}
So, $\alpha_{1}e_{1,m}+\alpha_{2}e_{2,m}+\dotsb+\alpha_{k}e_{k,m} = 0$ for each $m \in \{1,2,\dotsc,d\}$.  As such,
\begin{equation*}
0 = \sum_{j=1}^{d}\left(\sum_{i=1}^{k}\alpha_{i}e_{i,j}\right) v_{j} = \sum_{i=1}^{k} \alpha_{i}\sum_{j=1}^{d} e_{i,j}v_{j} = \sum_{i=1}^{k} \alpha_{i}\mathscr{L}_{n_{i}}(\mathbf{v})
\end{equation*}
for all $\mathbf{v} = (v_{1},v_{2},\dotsc,v_{d}) \in \mathbb{F}_{q}^{d}$.  That is, $\mathscr{L}_{n_{1}},\mathscr{L}_{n_{2}},\dotsc,\mathscr{L}_{n_{k}}$ are $\mathbb{F}_{q}$-dependent. \\
\item Suppose $k < \frac{\log q}{10\log_{2}q}$.  By
  \ref{lemma:2.2(a)}, it suffices to show that multiplicative
  independence implies $\mathbb{F}_{q}$-independence.  Suppose that
  $n_{1},n_{2},\dotsc,n_{k}$ are multiplicatively independent.  If we let $E
  := (e_{i,j})_{i=1,2,\dotsc,k}^{j=1,2,\dotsc,d}$, then
  $\rk_{\mathbb{Z}}(E) = \rk_{\mathbb{Q}}(E) = k$.  In particular,
  there exists an invertible $k \times k$ matrix $E^{\prime}$ which
  consists of $k$ independent columns of $E$.  Without loss of
  generality, the first $k$ columns of $E$ are independent.  Suppose
  $\mathscr{L}_{n_{1}},\mathscr{L}_{n_{2}},\dotsc,\mathscr{L}_{n_{k}}$
  are $\mathbb{F}_{q}$-dependent.  Let $\alpha =
  (\alpha_{1},\alpha_{2},\dotsc,\alpha_{k}) \in \mathbb{F}_{q}^{k}
  \setminus \{0\}$ be such that
  $\alpha_{1}\mathscr{L}_{n_{1}}+\alpha_{2}\mathscr{L}_{n_{2}}+\dotsb+\alpha_{k}\mathscr{L}_{n_{k}}
  = 0$.  Then, $E^{\prime}\alpha = 0$.  In particular,
  $q\mid\det(E^{\prime})$.  Recall that Hadamard's inequality states 
 \begin{equation*}
\left|\det(E^{\prime})\right| \le \prod_{j=1}^{k} \|\mathbf{e}_{j}\|,
\end{equation*}
where $\mathbf{e}_{j}$ is the $j$\textsuperscript{th} row of
$E^{\prime}$ and $\|\cdot\|$ is the Euclidean norm   (e.g., see
\cite[\S2.11]{beckenbach}.)
Note that
$\|\mathbf{e}_{j}\| \le k^{1/2 }\frac{\log q}{\log p_{j}}$.  Thus, 
\begin{equation*}
q \text { divides } \left|\det(E^{\prime})\right| \le k^{k/2}\prod_{j=1}^{k} \frac{\log q}{\log p_{j}} < \frac{k^{k/2}(\log q)^{k}}{\log 2} < q
\end{equation*}
since $k < \frac{\log q}{10\log_{2}q}$.  Thus, $\det(E^{\prime}) = 0$,
which implies $\rk(E^{\prime}) < k$, which is a contradiction.  So, no
such $\alpha$ exists and the forms
$\mathscr{L}_{n_{1}},\mathscr{L}_{n_{2}},\dotsc,\mathscr{L}_{n_{k}}$
are $\mathbb{F}_{q}$-independent. 
\end{enumerate}
\end{proof}

\section{Proof of Theorem \ref{thm:3.1}}

To simplify the notation we will denote $\mathscr{N} :=
\mathscr{N}_{q}$, and let
\begin{equation}
  \label{eq:N-cardinality-def}
N := \# \mathscr{N}.
\end{equation}

\subsection{The subset $\mathscr{N}$}\label{subsec:3.2}

Recall the following notation: for $m \in \mathbb{N}$ and $\ell$ a fixed prime,
\begin{align*}
P(m) &:= \max\{p: p|m\} && (\text{the {largest prime divisor} of } m) \\
\nu_{\ell}(m) &:= \max\{\alpha \in \mathbb{N} \cup \{0\}: \ell^{\alpha} | m\} && (\text{the } \ell \text{-adic valuation} \text{ of } m)
\end{align*}
The following parameters will be determined later:  $B$, respectively
$f(q)$, are parameters giving bounds on the exponents of large,
respectively small,
primes dividing elements of $\mathscr{N}$. 

Let
\begin{equation*}
\mathscr{N} := \Bigl\{n \in \{2,3,\dotsc,q-1\}:  n = sr, \text{ where } s \in \mathscr{S} \text{ and } r \in \mathscr{R} \Bigr\},
\end{equation*}
where
\begin{equation*}
\mathscr{S} := \Bigl\{s \in \{1,2,\dotsc,q-1\}: P(s) \le B \text{ and } \nu_{p}(s) \le f(q)\text{ for all primes } p\Bigr\}
\end{equation*}
and
\begin{equation*}
\mathscr{R} := \Bigl\{r \in \{1,2,\dotsc,q-1\}: p|r \text{ and } p \ge B \text{ implies } p\| r \Bigr\}.
\end{equation*}
We then find (recall that $N = \# \mathscr{N}$,
cf. \eqref{eq:N-cardinality-def}) 
\begin{align*}
q-2-N &\le \#\Bigl\{n < q: p \mid n \text{ for some } p \le B \text{ and } \nu_{p}(n) \ge f(q)\Bigr\} \\
&\qquad + \#\Bigl\{n < q:  p^{2} \mid n \text{ for some } p > B\Bigr\}.
\end{align*}
These quantities can be bounded as follows:
\begin{align*}
&\#\Bigl\{n < q: p \mid n \text{ for some } p \le B \text{ and } \nu_{p}(n) \ge f(q)\Bigr\} \\
&\qquad \le \sum_{p \le B} \#\Bigl\{n < q: p \mid n \text{ implies } \nu_{p}(n) \ge f(q)\Bigr\} \\
&\qquad \le q\sum_{p \le B} \frac{1}{p^{f(q)}} \le \frac{q}{2^{f(q)}} \pi(B).
\end{align*}
and
\begin{equation*}
\#\Bigl\{n < q: p^{2} \mid n \text{ for some } p > B\Bigr\} 
\le
\sum_{p > B} \frac{q}{p^{2}} \le \frac{3q}{B\log B} 
\end{equation*}
%By \cite[Exercise 5.5.3]{Cojocaru&Murty1},
%\begin{equation*}
%\sum_{p > B} \frac{1}{p^{2}} = -\frac{\pi(B)}{B^{2}}+2\int_{B}^{\infty} \frac{\pi(t)}{t^{3}} \,\, dt \le -\frac{1}{B\log B}+\frac{4}{\log B} \int_{B}^{\infty} \frac{1}{t^{2}} \,\, dt = \frac{3}{B\log B}
%\end{equation*}
for all $B \ge 1$.  In particular,
\begin{equation*}
q-2-N 
\le 
\frac{q\pi(B)}{2^{f(q)}}+\frac{3q}{B\log B} 
\le
\frac{qB}{2^{f(q)-1}\log B}+\frac{3q}{B\log B}. 
\end{equation*}

Define 
\begin{equation}
  \label{eq:define-B-and-fq}
B := c_{1}\log_{2}q, \quad f(q) := c_{2}\log_{3}q,
\end{equation}
where  $c_{1}, c_{2}>0$ are constants to be chosen later.  Then, 
\begin{align*}
q-2-N &\le \frac{c_{1}q\log_{2}q}{2^{c_{2}(\log\log\log q)-1}}+\frac{3q}{c_{1}(\log_{2}q)(\log c_{2}+\log_{3}q)} \\
&\le 2c_{1}q\exp\left(\log_{3}q - (\log_{3}q)^{c_{2}\log 2}\right)+\frac{3q}{c_{1}(\log_{2}q)(\log c_{2}+\log_{3}q)} \\
&\le c_{3}\frac{q}{\log_{2}q},
\end{align*}
where $c_{3}$ is a constant and $c_{3} \in (0,1)$ if $c_{2} > 2/\log
2$.  In particular, for $c_{2} > 2/\log 2$,
\begin{equation}\label{eq:3.4}
N = q+O\left(\frac{q}{\log_{2}q}\right),
\end{equation}
where the implied constant in \eqref{eq:3.4} is less than $1$.
%where $f(x) = O^{*}(g(x))$ or $f(x) \ll^{*} g(x)$ indicates that the implied constant is $\le 1$.

\subsection{Multiplicatively dependent $k$-tuples of $\mathscr{N}$}\label{subsec:3.3}

Assume that we are given
distinct multiplicatively
dependent integers $n_{1},n_{2},\dotsc,n_{k} \in \mathscr{N}$,
and suppose that $r < k$ is the (multiplicative) rank of these integers.  That
is, there exists 
$m_{1}, m_{2}, \dotsc, m_{r} \in \{n_{1},n_{2},\dotsc,n_{k}\}$ such
that
\begin{enumerate}[label=\textnormal(\alph{enumi})]
\item $m_{1}, m_{2}, \dotsc, m_{r}$ are multiplicatively independent and
\item for any $n \in \{n_{1},n_{2},\dotsc,n_{k}\} \setminus
  \{m_{1},m_{2},\dotsc,m_{r}\}$, the enlarged set $\{m_{1},m_{2},\dotsc,m_{r},
  n\}$ is multiplicatively dependent. 
\end{enumerate}
Without loss of generality, $n_{i} = m_{i}$ for all $i \in
\{1,2,\dotsc,r\}$.  Then, for every $j \in \{r+1,r+2,\dotsc,k\}$,
there exists $\alpha_{j} \in \mathbb{N}$ and
$\alpha_{1j},\alpha_{2j},\dotsc,\alpha_{rj} \in \mathbb{Z}$ such that 
\begin{equation*}
n_{j}^{\alpha_{j}} = n_{1}^{\alpha_{1j}}n_{2}^{\alpha_{2j}}\dotsm n_{r}^{\alpha_{rj}}.
\end{equation*}
For convenience, let $j = r + 1$, $\alpha=\alpha_{r+1}$, $\alpha_{ij} = \alpha_{i}$ and $n = n_{r+1}$.  Then,
\begin{equation}\label{eq:3.5}
n^{\alpha} = n_{1}^{\alpha_{1}}n_{2}^{\alpha_{2}}\dotsm n_{r}^{\alpha_{r}}.
\end{equation}
Let $J_{+}=\{j \in \{1,2,\dotsc,r\}: \alpha_{j} > 0\}$, $J_{-}=\{j \in \{1,2,\dotsc,r\}: \alpha_{j} < 0\}$ and $J_{0}=\{j \in \{1,2,\dotsc,r\}: \alpha_{j} = 0\}$.  Note that $|J_{+}|+|J_{-}|+|J_{0}| = r$ and
\begin{equation*}
\rad(n) \,\bigg|\, \rad\left(\prod_{j \in J_{+}}n_{j}\right)
\end{equation*}
as $n_{1},n_{2},\dotsc,n_{r} \in \mathbb{N}$.  

\textbf{\underline{Case 1}}: $|J_{-}|=0$.  In
this case,
$\rad(n_{j}) \in \{d \in \mathbb{N}: d \mid \rad(n)\}$.  Thus, there
are $\tau(\rad(n))^{|J_{+}|}$ choices for the radicals of elements
corresponding to $J_{+}$.  There are also $N^{|J_{0}|}$ choices for
elements corresponding to $J_{0}$.

For any squarefree number $n_{0} \in \mathscr{N}$, the number of
elements in $m \in \mathscr{N}$ with radical $n_{0}$ is bounded as
follows:  recall $m \in \mathscr{N}$ satisfies the condition that
$\nu_{p}(m) \le 1$ for all $p > B$.  So, the only place where
$\rad(m)=n_{0}$ and $m$ differ is in the prime factors $p \le B$.
Thus, by the definition of $\mathscr{N}$, the number of choices for the
difference of $m$ and $n_{0}$ is bounded by (recall \eqref{eq:define-B-and-fq})
\begin{equation*}
\pi(B)^{f(q)} \ll \left(\frac{\log_{2}q}{\log_{3}q}\right)^{c_{2}\log_{3}q} \ll \exp\left(c_{2}(\log_{3}q)^{2}\right).
\end{equation*}
So, the number of choices for $n_{j}$ with $j \in J_{+}$ corresponding
to $n$ is 
\begin{equation*}
\ll
\tau(\rad(n))^{|J_{+}|}\exp\left(c_{2}|J_{+}|(\log_{3}q)^{2}\right).
\end{equation*}
The classical bound 
\begin{equation}
  \label{eq:tau-bound}
\tau(m) \ll \exp(C_{\tau}\log m/\log_{2}m),
\end{equation}
where $C_{\tau}>0$ is a constant, yields 
\begin{equation*}
\tau(\rad(n)) \ll \exp\left(C_{\tau}\frac{\log q}{\log_{2}q}\right).
\end{equation*}
As such, the number of choices of $n_{j}$ with $j \in J_{+}$ given $n$ is
\begin{equation*}
\ll \exp\left(C_{\tau}|J_{+}|\frac{\log q}{\log_{2}q}+c_{2}|J_{+}|(\log_{3}q)^{2}\right).
\end{equation*}
So, the number of overall choices is
% for $n_{j}$ with $j \in J_{+}$ and $n=n_{r+1}$ is bounded by 
%                                
\begin{equation*}
\ll N^{1+|J_{0}|}\exp\left(C_{\tau}|J_{+}|\frac{\log q}{\log_{2}q}+c_{2}|J_{+}|(\log_{3}q)^{2}\right).
\end{equation*}

\textbf{\underline{Case 2}}: $|J_{-}| > 0$.  Then,
\begin{equation*}
n\prod_{j \in J_{-}} n_{j}^{\alpha_{j}} = \prod_{j \in J_{+}}n_{j}^{\alpha_{j}} =: n_{J_{+}}
\end{equation*}
Let $m = \rad(n_{J_{+}})$.  Hence, $\rad(n) \mid \rad(m)$ and $\rad(n_{j}) \mid \rad(m)$ for all $j \in J_{-}$.  As before, the number of choices for the radical of $n, n_{j}$ with $j \in J_{-}$ is bounded by $\tau(m)^{1 + |J_{-}|}$.  Also, from the computation for the number of elements in $\mathscr{N}$ with radical $m$, we have that the number of overall choices in this case is bounded by
\begin{align*}
&\ll N^{|J_{0}|+|J_{+}|}\exp\left(|J_{-}|C_{\tau}(r - 1)\frac{\log q}{\log_{2} q} + c|J_{-}|(\log_{3} q)^{2}\right) \\
&\ll N^{r-|J_{-}|} \exp\left(C_{\tau}r^{2}\frac{\log q}{\log_{2} q} + cr(\log_{3} q)^{2}\right).
\end{align*}

Note that, for the remaining elements $m_{1}$ in
$\{n_{r+2},n_{r+3},\dotsc,n_{k}\}$ in Cases 1 and 2, we have 

\begin{equation*}
\rad(m_{1}) \mid \rad\left(\prod_{i=1}^{r} n_{i}\right).
\end{equation*}
In particular, there are $\tau(n_{1}n_{2}\dotsm n_{r})^{k-r-1}$
choices for the radical of $m_{1}$.  
%
%This quantity can be bounded by
Using the previous bound on the number of ways an element in
$\mathscr{N}$ can have a fixed radical, we find that the total number
of ways to chose the remaining $n_{r+2},n_{r+3},\dotsc,n_{k}$ is 
\begin{equation*}
\ll \exp\left(k^{2}C_{\tau}\frac{\log q}{\log_{2} q}\right).
\end{equation*}

From the bounds in the two different cases it follows that the number of
distinct $k$-tuples of elements in $\mathscr{N}$, having
multiplicative rank $r < k$, and $J_{-},J_{0},J_{+}$ fixed, is 
\begin{multline}
\label{eq:case12}
\ll 
\begin{cases}
N^{1+|J_{0}|}\exp\left(C_{\tau}(k^{2}+|J_{+}|)\frac{\log
    q}{\log_{2}q}+c_{2}|J_{+}|(\log_{3}q)^{2}\right) 
& \text{if $|J_{-}|=0$,}
\\  
 N^{r-|J_{-}|} \exp\left(C_{\tau}(k^{2}+r^{2})\frac{\log q}{\log_{2}
    q} + cr(\log_{3} q)^{2}\right)
&\text{if $|J_{-}|>0$.}
\end{cases}
\end{multline}

% So our bounds become
% \begin{align*}
% \textbf{\underline{Case 1}}:  &\ll N^{1+|J_{0}|}\exp\left(C_{\tau}(k^{2}+|J_{+}|)\frac{\log q}{\log_{2}q}+c_{2}|J_{+}|(\log_{3}q)^{2}\right) \\
% \textbf{\underline{Case 2}}:  &\ll N^{r-|J_{-}|} \exp\left(C_{\tau}(k^{2}+r^{2})\frac{\log q}{\log_{2} q} + cr(\log_{3} q)^{2}\right).
% \end{align*}

We claim that (\ref{eq:case12}) is $O(N^{r-1/2+o(1)})$ for suitably small $r$;
this clearly holds if $|J_{-}|>0$.
If $|J_{-}|=0$ the
supposition would hold if $|J_{0}| < r-1$.  Suppose, on the contrary,
that $|J_{0}| = r-1$.  Then, \eqref{eq:3.5} yields
$n_{r+1}^{\alpha_{r+1}} = n_{i}^{\alpha_{i}}$ for some $i \leq r$, and
$\alpha_{r+1}, \alpha_{i}>0$.
Without loss of generality we can assume that $(\alpha_{r+1},
\alpha_{i})=1$, and $\alpha_{r+1}>\alpha_{i}$ (the case
$\alpha_{r+1}<\alpha_{i}$ is similar.)
%%  Moreover, if $S=\{n_{1},n_{2},\dotsc,n_{k}\}$ were to contribute to
%%  $M_{q,S}$ (that is, $\mathbf{v} \in \mathbb{F}_{q}^{d}$ such that
%%  $\mathscr{L}_{n_{r+1}}(\mathbf{v})=1=\mathscr{L}_{n_{i}}(\mathbf{v})$), then
%%  $\alpha_{r+1} \equiv \alpha_{i} \Mod{q}$.  In particular, since both
%%  $\alpha_{r+1}$ and $\alpha_{i}$ are positive, without loss of
%%  generality $\alpha_{r+1}=\alpha_{i}+qs$ for some $s \in \mathbb{N}$
%%  (note that $s \neq 0$ since $n_{r+1} \neq n_{i}$.)  
Since $(\alpha_{r+1}, \alpha_{i})=1$ we must have
$n_{r+1} = M^{\alpha_{i}}$ and $n_{i} = M^{\alpha_{r+1}}$ for some
integer $M>1$; as $\alpha_{r+1} \ge 2$ and $n_{i} < q $, there are at
most $q^{1/2}$ choices for $M$, and consequently there are a total of
$O(q^{1/2+o(1)})$ choices for $n_{r+1}$ and $n_{i}$, and at most
$N^{r-1}$ choices for the remaining $n_{1},n_{2},\ldots, n_{i-1},n_{i+1}, n_{r}$.

Thus (\ref{eq:case12}) is $O(N^{r-1/2+o(1)})$ if $r$ is sufficiently
small, and since $r \le k$, a choice of $k=o(\sqrt{\log_{2}q})$ will
suffice.  
For more explicit error terms 
we will argue as follows.
Recall that an initial choice of a basis of size $r$ was
chosen.
Now, for $r$ fixed, the number of possible choices of triples,
$J_{+}, J_{-}$ and $J_{0}$ are bounded by the combinatorial factor 
$3^{r}$.
We thus find that
\begin{equation}\label{eq:3.7}
\#\bigl\{\mathscr{S} \in \mathscr{N}^{[k]}: \rk(\mathscr{S})=r\bigr\} \ll \binom{k}{r}3^{r}N^{r-1/2+o(1)}\exp\left(2C_{\tau}k^{2}\frac{\log q}{\log_{2}q} + c_{2}r(\log_{3}q)^{2}\right),
\end{equation}
%where, for a set $\mathscr{A}$ and $k \in \mathbb{N}$, $\mathscr{A}^{[k]} := \{\mathscr{B} \subset \mathscr {A}: |\mathscr{B}|=k\}$.
and hence
%\begingroup
%\allowdisplaybreaks
\begin{align*}
&\#\bigl\{\mathscr{S} \in \mathscr{N}^{[k]}: \mathscr{S} \text{ is multiplicatively dependent} \bigr\} \\
%&\qquad = \sum_{r=2}^{k-1} \#\bigl\{\mathscr{S} \in \mathscr{N}^{[k]}: \rk(\mathscr{S})=r\bigr\} \\
&\qquad \ll \sum_{r=1}^{k-1} \binom{k}{r}3^{r}N^{r-1/2+o(1)}\exp\left(2C_{\tau}k^{2}\frac{\log q}{\log_{2}q} + c_{2}r(\log_{3}q)^{2}\right) \\
&\qquad \ll N^{k-3/2+o(1)}\exp\left(2C_{\tau}k^{2}\frac{\log q}{\log_{2}q}+c_{2}k(\log_{3}q)^{2}\right) \sum_{r=0}^{\infty} \frac{(3k)^{r}}{r!} \\
&\qquad \ll N^{k-3/2+o(1)}\exp\left(3k+2C_{\tau}k^{2}\frac{\log q}{\log_{2}q}+c_{2}k(\log_{3}q)^{2}\right).
\end{align*}
%\endgroup
In particular, for $k = o(\sqrt{\log_{2}q})$,
\begin{align*}%\label{eq:3.8}
&\#\left\{\mathscr{S} \in \mathscr{N}^{[k]}: \mathscr{S} \text{ is multiplicatively independent}\right\} \\
&\qquad =
  \binom{N}{k}+O\left(N^{k-3/2+o(1)}\exp\left(3k+2C_{\tau}k^{2}\frac{\log
  q}{\log_{2}q}+c_{2}k(\log_{3}q)^{2}\right)\right) \notag
\\
&\qquad= \binom{N}{k} + O 
  \left( N^{k-3/2+o(1)})   \right), \notag
\end{align*}
thus proving the first equality in \eqref{eq:better-estimate}.
Then, \eqref{eq:3.4} and the comment following it imply, for $k =
o(\log_{2} q)$, that
\begin{multline}
\label{eq:3.9}
\binom{N}{k} = \frac{1}{k!}N(N-1)(N-2)\dotsm(N-k+1)
%&= \frac{1}{k!}\left(q+O^{*}\left(\frac{q}{\log_{2}q}\right)\right)\left(q+O^{*}\left(\frac{q}{\log_{2}q}\right)-1\right) \dotsm \left(q+O^{*}\left(\frac{q}{\log_{2}q}\right)-k\right) \notag \\
= \frac{1}{k!}\left(q+O\left(\frac{q}{\log_{2}q}\right)\right)^{k}\\
%&=
%\frac{q^{k}}{k!}+O^{*}\left(\binom{k}{1}\frac{q^{k}}{\log_{2}q}+\binom{k}{2}\frac{q^{k}}{(\log_{2}q)^{2}}+\dotsb+\binom{k}{k}\frac{q^{k}}{(\log_{2}q)^{k}}\right)
%\notag \\
=
\frac{q^{k}}{k!} (1 + O(1/\log_2 q))^{k}
=
\frac{q^{k}}{k!} (1 + O(k/\log_2 q))
=
\frac{q^{k}}{k!}+O\left(\frac{q^{k}}{(k-1)!\log_{2}q}\right).
%&= \frac{q^{k}}{k!}+O\left(\frac{q^{k}}{k!}\left(\binom{k}{1}\frac{1}{\log_{2}q}+\binom{k}{2}\frac{1}{(\log_{2}q)^{2}}+\dotsb+\binom{k}{k}\frac{1}{(\log_{2}q)^{k}}\right)\right). \notag
\end{multline}
%% If $k$ is odd, then
%% \begin{align*}
%% &\binom{k}{1}\frac{1}{\log_{2}q}+\binom{k}{2}\frac{1}{(\log_{2}q)^{2}}+\dotsb+\binom{k}{k}\frac{1}{(\log_{2}q)^{k}} \\
%% &\qquad \quad \le 2\left(\binom{k}{1}\frac{1}{\log_{2}q}+\binom{k}{2}\frac{1}{(\log_{2}q)^{2}}+\dotsb+\binom{k}{\frac{k-1}{2}}\frac{1}{(\log_{2}q)^{\frac{k-1}{2}}}\right) \\
%% &\qquad \quad \le 2\left(\frac{k}{\log_{2}q}+\frac{k^{2}}{2(\log_{2}q)^{2}}+\dotsb+\frac{k^{(k-1)/2}}{\left(\frac{k-1}{2}\right)!(\log_{2}q)^{\frac{k-1}{2}}}\right) \\
%% &\qquad \quad \le 2\left(\exp(k/\log_{2}q)-1\right)
%% \end{align*}
%% and if $k$ is even, the same bound holds.
%% %\begin{align*}
%% %&\binom{k}{1}\frac{1}{\log_{2}q}+\binom{k}{2}\frac{1}{(\log_{2}q)^{2}}+\dotsb+\binom{k}{k}\frac{1}{(\log_{2}q)^{k}} \\
%% %&\qquad \quad \le 2\left(\binom{k}{1}\frac{1}{\log_{2}q}+\binom{k}{2}\frac{1}{(\log_{2}q)^{2}}+\dotsb+\binom{k}{\frac{k}{2}}\frac{1}{(\log_{2}q)^{\frac{k}{2}}}\right) \\
%% %&\qquad \quad \le 2\left(\frac{k}{\log_{2}q}+\frac{k^{2}}{2(\log_{2}q)^{2}}+\dotsb+\frac{k^{k/2}}{\frac{k}{2}!(\log_{2}q)^{\frac{k}{2}}}\right) \\
%% %&\qquad \quad \le 2\left(\exp\left(\frac{k}{\log_{2}q}\right)-1\right)
%% %\end{align*}
%% Suppose $k < \log_{2}q$.  Then,
%% \begin{equation*}
%% \binom{k}{1}\frac{1}{\log_{2}q}+\binom{k}{2}\frac{1}{(\log_{2}q)^{2}}+\dotsb+\binom{k}{k}\frac{1}{(\log_{2}q)^{k}} \le \frac{k}{\log_{2}q}.
%% \end{equation*}
%Hence, by \eqref{eq:3.9}
%\begin{equation*}
%\binom{N}{k} = \frac{q^{k}}{k!}+O\left(\frac{q^{k}}{(k-1)!\log_{2}q}\right),
%\end{equation*}
%where the implied constant is less than $1$. \\
Moreover, if $k=o(\sqrt{\log_{2} q})$, then
%< c\log_{2}q$, where $c < \max\{1,1/C_{\tau}\}$, then \eqref{eq:3.8} yields
\begin{align*}
&\#\left\{\mathscr{S} \in \mathscr{N}^{[k]}: \mathscr{S} \text{ is
  multiplicatively independent}\right\} \\
&\qquad = \frac{q^{k}}{k!}+O\left(\frac{q^{k}}{(k-1)!\log_{2}q}\right)+O\left(N^{k-2}\exp\left(3k+2C_{\tau}k^{2}\frac{\log q}{\log_{2}q}+c_{2}k(\log_{3}q)^{2}\right)\right) \\
&\qquad = \frac{q^{k}}{k!}+O\left(\frac{q^{k}}{(k-1)!\log_{2}q}\right),
\end{align*}
where the implied constant is absolute. 
%%  As above,
%%  \begin{align*}
%%  \binom{q}{k} &= \frac{1}{k!}q(q-1)(q-2)\dotsm(q-k+1)=\frac{1}{k!}\left(q+O\left(\frac{q}{\log_{2}q}\right)\right)^{k} \\
%%  &= \frac{q^{k}}{k!}+O\left(\frac{q^{k}}{(k-1)!\log_{2}q}\right).
%%  \end{align*}
The proof of  Theorem \ref{thm:3.1} is thus concluded.

\section{Proof of Theorem \ref{thm:1.1}}

Denote  $\mathscr{N} = \mathscr{N}_{q}$ and $N = \#\mathscr{N}$, with
$\mathscr{N}_{q}$ as in 
Theorem~\ref{thm:3.1}.
Recall from \S \ref{subsec:1.3} that $N_{q} \leq  M_{q,\mathscr{N}}$,
and that
\begin{equation*}
\sum_{k=0}^{2K-1} (-1)^{k}\sum_{\substack{S \subset \mathscr{N} \\ |S|
    = k}} M_{q,S} 
%\le N_{q} 
\le M_{q,\mathscr{N}}
\le \sum_{k=0}^{2K} (-1)^{k}
\sum_{\substack{S \subset \mathscr{N} \\ |S| = k}} M_{q,S}, 
\end{equation*}
where
%\begin{equation*}
$M_{q,S} := \#\{\mathbf{v} \in \mathbb{F}_{q}^{d}: \mathscr{L}_{n}(\mathbf{v}) = 1 \text{ for all } n \in S\}.$
%\end{equation*}
%
Let
\begin{equation*}
\Sigma := \Sigma_{K} := \sum_{k=0}^{K} (-1)^{k}\sum_{\substack{S \subset \mathscr{N} \\ |S| = k}} M_{q,S}.
\end{equation*}
Then,
\begin{align}
\label{eq:sum1=sum2+sum3}
\begin{split}
\Sigma &= \sum_{k = 0}^{K} (-1)^{k}
         \sum_{\substack{S=\{n_{1},n_{2},\dotsc,n_{k}\} \in \mathscr{N}^{[k]}\\
%         \{2,3,\dotsc,q-1\} \\
  \{\mathscr{L}_{n_{1}},\mathscr{L}_{n_{2}},\dotsc,\mathscr{L}_{n_{k}}\}
  \text{ is } \mathbb{F}_{q}\text{-independent}}} M_{q,S} \\
&\qquad + \sum_{k = 0}^{K} (-1)^{k}
  \sum_{\substack{S=\{n_{1},n_{2},\dotsc,n_{k}\} \in
\mathscr{N}^{[k]}\\
%  \{2,3,\dotsc,q-1\} \\
  \{\mathscr{L}_{n_{1}},\mathscr{L}_{n_{2}},\dotsc,\mathscr{L}_{n_{k}}\}
  \text{ is } \mathbb{F}_{q}\text{-dependent}}} M_{q,S} \\
&=: \Sigma_{1} + \Sigma_{2},
\end{split}
\end{align}
say.  Now, for $K < \log q/(10\log_{2} q)$,  Lemma
\ref{lemma:2.2} together with the rank-nullity theorem of linear
algebra implies that 
\begin{equation}\label{eq:3.2}
\Sigma_{1} = \sum_{k = 0}^{K}
(-1)^{k}q^{d-k}\sum_{\substack{ \{n_{1},n_{2},\dotsc,n_{k}\} \in
    \mathscr{N}^{[k]} \\ n_{1},n_{2},\dotsc,n_{k} \text{ are} \\ \text{
      multiplicatively independent}}} 1 
\end{equation}
%(the main term) 
and
\begin{equation}\label{eq:3.3}
\Sigma_{2} = \sum_{k = 0}^{K}
(-1)^{k}\sum_{\substack{S = \{n_{1},n_{2},\dotsc,n_{k}\} \in \mathscr{N}^{[k]} \\
    n_{1},n_{2},\dotsc,n_{k} \text{ are} \\ \text{ multiplicatively
      dependent}}} M_{q,S}. 
\end{equation}
%(an error term.)
% In order to evaluate the main term \eqref{eq:3.2}, we need to show that
% \begin{equation*}
% \sum_{\substack{n_{1},n_{2},\dotsc,n_{k} \in \mathscr{N} \\
%     n_{1},n_{2},\dotsc,n_{k} \text{ are} \\ \text{ multiplicatively
%       independent}}} 1 = 
% %\binom{q}{k} + o\left(\binom{q}{k}\right).
% \frac{q^{k}}{k!}(1+o(1)).
% \end{equation*}

\subsection{The Upper Bound}\label{subsec:3.4}

For $K$ even (recall \eqref{eq:sum1=sum2+sum3}-\eqref{eq:3.3}), we have
\begin{align*}
N_{q} \le \#\left\{\mathbf{v} \in \mathbb{F}_{q}^{d}:
  \mathscr{L}_{n}(\mathbf{v}) \ne 1 \text{ for all } n \in \mathscr{N}
  \right\} \le \Sigma_{1}+\Sigma_{2}.
\end{align*}
%%%%   where, for $K$ even,
%%%%   \begin{equation*}
%%%%   \Sigma_{1} := \sum_{k = 0}^{K} (-1)^{k}q^{d-k}\sum_{\substack{n_{1},n_{2},\dotsc,n_{k} \in \mathscr{N} \\ n_{1},n_{2},\dotsc,n_{k} \text{ are} \\ \text{multiplicatively independent}}} 1
%%%%   \end{equation*}
%%%%   and
%%%%   \begin{equation*}
%%%%   \Sigma_{2} := \sum_{k = 0}^{K} (-1)^{k}\sum_{\substack{\substack{n_{1},n_{2},\dotsc,n_{k} \in \mathscr{N} \\ n_{1},n_{2},\dotsc,n_{k} \text{ are} \\ \text{multiplicatively dependent}}}} M_{q,S}.
%%%%   \end{equation*}
Theorem \ref{thm:3.1}, together with (\ref{eq:3.2}), gives
\begin{align}\label{eq:3.10}
\Sigma_{1} &= \sum_{k=0}^{K} (-1)^{k}q^{d-k}\left(\frac{q^{k}}{k!}+O\left(\frac{q^{k}}{(k-1)!\log_{2}q}\right)\right) \\%=q^{d}\sum_{k=0}^{K} \frac{(-1)^{k}}{k!} + O\left(\frac{q^{d}}{\log_{2}q}\sum_{k=0}^{K} k\right) \notag \\
&= \frac{q^{d}}{e}+O\left(q^{d} \sum_{k > K} \frac{1}{k!}\right)+O\left(\frac{q^{d}}{\log_{2}q}\right) \notag \\
&=\frac{q^{d}}{e}+O\left(\left(\frac{2}{K}\right)^{\frac{K}{2}}q^{d}\right)+O\left(\frac{q^{d}}{\log_{2}q}\right) \notag
\end{align}
for $K$ growing with $q$ so that
$K=o\left(\sqrt{\log_{2}q} \right)$.
%, the above is an asymptotic formula.

For $\Sigma_{2}$, there are no multiplicatively dependent sets of size
$1$ unless $n = 1$.  By \eqref{eq:3.7},
%and (\ref{eq:3.3}) (recall that $K$ is even), 
together with the rank-nullity theorem, we have
\begin{align}\label{eq:3.11}
\Sigma_{2} &\ll \sum_{k=2}^{K}
             \sum_{{\substack{S=\{n_{1},n_{2},\dotsc,n_{k}\} \in
             \mathscr{N}^{[k]} 
  \\ n_{1},n_{2},\dotsc,n_{k} \text{ are} \\ \text{multiplicatively
  dependent}}}} M_{q,S} \le \sum_{k=2}^{K} \sum_{r=1}^{k-1}
  \sum_{\substack{S=\{n_{1},n_{2},\dotsc,n_{k}\} \in \mathscr{N}^{[k]}
  \\ \rk_{\mathbb{Z}}(S)=r}} M_{q,S} \\ 
&\ll \sum_{k=2}^{K} \sum_{r=1}^{k-1} \binom{k}{r}3^{r}q^{d-r}N^{r-1}\exp\left(2C_{\tau}r^{2}\frac{\log q}{\log_{2}q} + c_{2}r(\log_{3}q)^{2}\right) \notag \\
&\ll q^{d-1}\exp\left(3K +2C_{\tau}K^{2}\frac{\log
  q}{\log_{2}q}+c_{2}K(\log_{3}q)^{2}\right). \notag 
\end{align}
%This estimate is $o(q^{d})$ for $K = o(\sqrt{\log_{2}q})$.  
Thus, by
\eqref{eq:3.10} and \eqref{eq:3.11},% (again for $K = o(\sqrt{\log_{2}q})$),
\begin{align*}
N_{q}  \le M_{q,\mathscr{N}}
&\le \frac{q^{d}}{e}+O\left(\left(\frac{2}{K}\right)^{\frac{K}{2}}q^{d}\right)+O\left(\frac{q^{d}}{\log_{2}q}\right) \\
&\qquad +O\left(q^{d-1}\exp\left(3 K  + 2C_{\tau}K^{2}\frac{\log q}{\log_{2}q}+c_{2}K(\log_{3}q)^{2}\right)\right) \\
&= \frac{q^{d}}{e}+O\left(\left(\frac{2}{K}\right)^{\frac{K}{2}}q^{d}\right)+O\left(\frac{q^{d}}{\log_{2}q}\right)
\end{align*}
for $K=o\left( \sqrt{ \log_{2}q}\right)$,
%The optimal error term
%occurs when $K \asymp \frac{\log_{3}q}{\log_{4}q}$ and yields 
%and taking $K \asymp \frac{\log_{3}q}{\log_{4}q}$ yields 
and taking $K \asymp (\log_2 q)^{1/3}$ yields 
\begin{equation*}
M_{q,\mathscr{N}} \le \frac{q^{d}}{e}+
%O\left(\frac{q^{d}(\log_{3}q)^{2}}{(\log_{2}q)(\log_{4}q)^{2}}\right).
O\left(\frac{q^{d}}{\log_{2}q}\right).
\end{equation*}
A similar argument with $K$ odd gives a  lower bound of the same form, and thus
\begin{equation}
  \label{eq:MqN-asymptotics}
N_{q} \le M_{q,\mathscr{N}} 
=
\frac{q^{d}}{e}+
%O\left(\frac{q^{d}(\log_{3}q)^{2}}{(\log_{2}q)(\log_{4}q)^{2}}\right),  
O\left(\frac{q^{d}}{\log_{2}q}\right).
\end{equation}

% As $N_{q} \leq M_{q,\mathscr{N}}$,  the proof of the upper
% bound of Theorem \ref{thm:1.1} follows after dividing by $q^{d}$, 

\subsection{The Lower Bound}\label{subsec:3.5}

In \S\ref{subsec:3.4}, we proved
\begin{align*}
N_{q} &:= \#\bigl\{\mathbf{v} \in \mathbb{F}_{q}^{d}:
        \mathscr{L}_{n}(\mathbf{v}) \ne 1 \text{ for all } n \in
        \{2,3,\dotsc,q-1\}\bigr\} \\ 
&\le M_{q,\mathscr{N}} = 
  \frac{q^{d}}{e}+
%O\left(\frac{q^{d}(\log_{3}q)^{2}}{(\log_{2}q)(\log_{4}q)^{2}}\right) 
O\left(\frac{q^{d}}{\log_{2}q}\right)
\end{align*}
by restricting the set $\{2,3,\dotsc,q-1\}$ to $\mathscr{N}$ and
showing (cf. (\ref{eq:MqN-asymptotics}))
\begin{equation*}
\#\left\{\mathbf{v} \in \mathbb{F}_{q}^{d}:
  \mathscr{L}_{n}(\mathbf{v}) \ne 1 \text{ for all } n \in
  \mathscr{N}\right\} =
\frac{q^{d}}{e}+
%O\left(\frac{q^{d}(\log_{3}q)^{2}}{(\log_{2}q)(\log_{4}q)^{2}}\right). 
O\left(\frac{q^{d}}{\log_{2}q}\right).
\end{equation*}
Let $\mathscr{N}_{\text{bad}} = \{2,3,\dotsc,q-1\} \setminus
\mathscr{N}$.  Then, 
\begin{equation*}
N_{q} = 
\frac{q^{d}}{e} + 
%O\left(\frac{q^{d}(\log_{3}q)^{2}}{(\log_{2}q)(\log_{4}q)^{2}}\right)  \\
O\left(\frac{q^{d}}{\log_{2}q}\right) 
 + O\Bigl(\bigl\{\mathbf{v} \in \mathbb{F}_{q}^{d}:
  \mathscr{L}_{m}(\mathbf{v}) = 1 \text{ for some } m \in
  \mathscr{N}_{\text{bad}}\bigr\}\Bigr). 
\end{equation*}
Note that $N_{\text{bad}} := \#\mathscr{N}_{\text{bad}} \ll \frac{q}{\log_{2}q}$ by
\eqref{eq:3.4}.  The number of $\mathbf{w} \in \{\mathbf{v} \in
\mathbb{F}_{q}^{d}: \mathscr{L}_{m}(\mathbf{v}) = 1 \text{ for some }
m \in \mathscr{N}_{\text{bad}}\}$ is bounded by
\begin{equation*}
\ll \sum_{m \in \mathscr{N}_{\text{bad}}} \#\left\{v \in
  \mathbb{F}_{q}^{d}: \mathscr{L}_{m}(\mathbf{v}) = 1\right\} \ll
q^{d-1}N_{\text{bad}} \ll \frac{q^{d}}{\log_{2}q}.
%= o(q^{d}) 
\end{equation*}
In particular, $N_{q} = q^{d}/e + O(q^{d}/\log_2 q)$; after dividing by
$q^{d}$
%after dividing by $q^{d}$ 
the proof of Theorem
\ref{thm:1.1} is concluded.

%\newpage
\section{Statistics}\label{subsec:1.2}
We have compared the model introduced by Friedrichsen and Holden with
the data from the problem.  Below (cf. Figures 1 and 2) are the
histograms and the quantile-quantile plots for some seven and
ten-digits primes.  The quantile-quantile plots compare the
theoretical quantiles (red line, Gaussian with mean 0 and standard
deviation 1) with the observed ones (coloured dots) from our
experiment.  The data is broken up based on how large $\omega(p-1)$
is.  The datasets sizes are 7216 (seven) and 241\,148 (ten).  The red
curve in the histograms is the Gaussian with
mean and standard deviation $\mu$ and $\sigma$, respectively, as reported. \\
\indent The model for the problem is as follows: we wish to count

\begin{equation*}
F(p) := \#\Big\{x \in \{1,2,\dotsc,p - 1\}: x^{x} \equiv x \Mod{p}\Big\}.
\end{equation*}

Consider the following lemma:
\begin{lemma}\label{lemma:1.3}
Let $p$ be a prime, 
$y \in \{1,2,\dotsc,p(p-1)\}$ an integer such that $p \nmid y$, and
let $d$ be a divisor of $p-1$
%For all primes $p$ and any fixed pair $(d, y)$, where divisor $d|p -
%1$ and 
such that $\ord_{p}y\mid d$.  Then,
\begin{equation*}
\#\Big\{x \in \{1,2,\dotsc,p(p-1)\}: p \nmid x, x^{x} \equiv y \Mod{p}, \ord_{p}x = d\Big\}
= \frac{p-1}{d}.
\end{equation*}
\end{lemma}
\noindent
Here 
$\ord_p(y)$ denotes the multiplicative order of $y$ modulo $p$, i.e.,
the smallest integer $k>0$ such that $y^{k} \equiv 1 \mod p$.

This lemma implies that the number of solutions to $x^{x} \equiv x \Mod{p}$ with $1 \le x \le p(p-1)$ is
\begin{equation*}
(p - 1) \sum_{d|n} \frac{\varphi(d)}{d}
\end{equation*}

\noindent (see Friedrichesen and Holden \cite{FrHo}).  %From the above result,
The above result suggests that $F(p)$ should be distributed as a binomial random variable
with mean $\mu$ and variance $\sigma^{2}$, where
\begin{equation*}
\mu = \sum_{d|p-1} \frac{\varphi(d)}{d} \qquad \text{and} \qquad \sigma^{2} = \sum_{d|p-1} \frac{\varphi(d)(d-1)}{d^{2}}.
\end{equation*}
The histograms in Figure 1 represent the normalized statistic for
$F(p)$ according 
to this model.  
\begin{figure}[ht]
\begin{center}
\includegraphics[width=12.5cm]{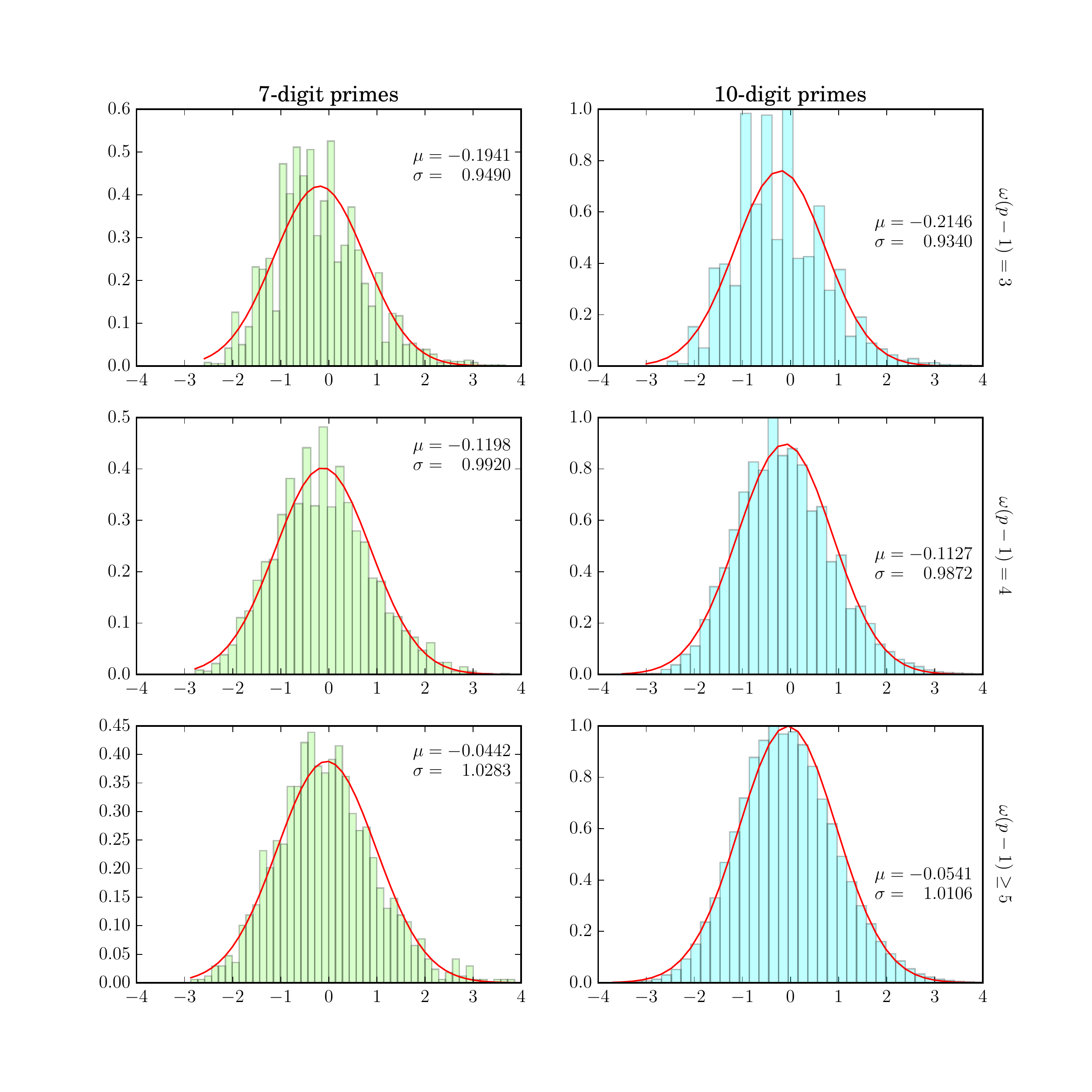}
\end{center}
\caption{Histograms for seven- (light green) and ten-digit (light
  blue) primes broken up into subgroups ($\omega(p-1)=3$,
  $\omega(p-1)=4$ and $\omega(p-1) \ge 5$).  The group $\omega(p-1)=2$
  was computed but the data gives rise to many outliers for reasons
  that are readily ascertainable.}
\end{figure}

That is, for a prime $p$, we compute $F(p)$ using
primitive roots 
and index calculus.  Then, we normalize $F(p)$ to $z = (F(p) - \mu)/\sigma$, where
$\mu$ and $\sigma$ are as above.  The resulting histograms are presented in Figure 1.
As can be seen from the histograms, the data seems to be tending to a normal distribution
$N(0, 1)$, especially in the mean $\mu$. \\
\indent The probability plots compare our observed data with the theoretical model $N(0, 1)$ as
follows.  The $i\textsuperscript{th}$ order (descending) statistic for the theoretical values is defined 
according to Filliben's estimate:
\begin{equation*}
i^{\text{th}} \text{ order statistic} =
\begin{cases}
(0.5)^{1/n} & \text{if } i=n, \\
\frac{i-0.3175}{n+0.365} \quad & \text{if } i \in \{2,3,\dotsc,n-1\} \\
1-(0.5)^{1/n} & \text{if } i=1,
\end{cases},
\end{equation*}
where $n$ is the size of the dataset.  As the quantile function is the inverse of the
cumulative distribution function, we obtain the red line in Figure 2.  For the observed
data, we sort the corresponding values for $z = (F(p)-\mu)/\sigma$ and plot these values
according to their values on the $y$-axis (observed values).  The high values of $R^{2}$
in Figure 2 indicate the model explains the observed variation very well.\\
\indent We note that, as can be seen in all the probability plots, there is a tendency for
the data to have a higher standard deviation on the tails.  We have not been able to determine
a satisfactory explanation for this behaviour.

\begin{figure}
\begin{center}
\includegraphics[width=12.5cm]{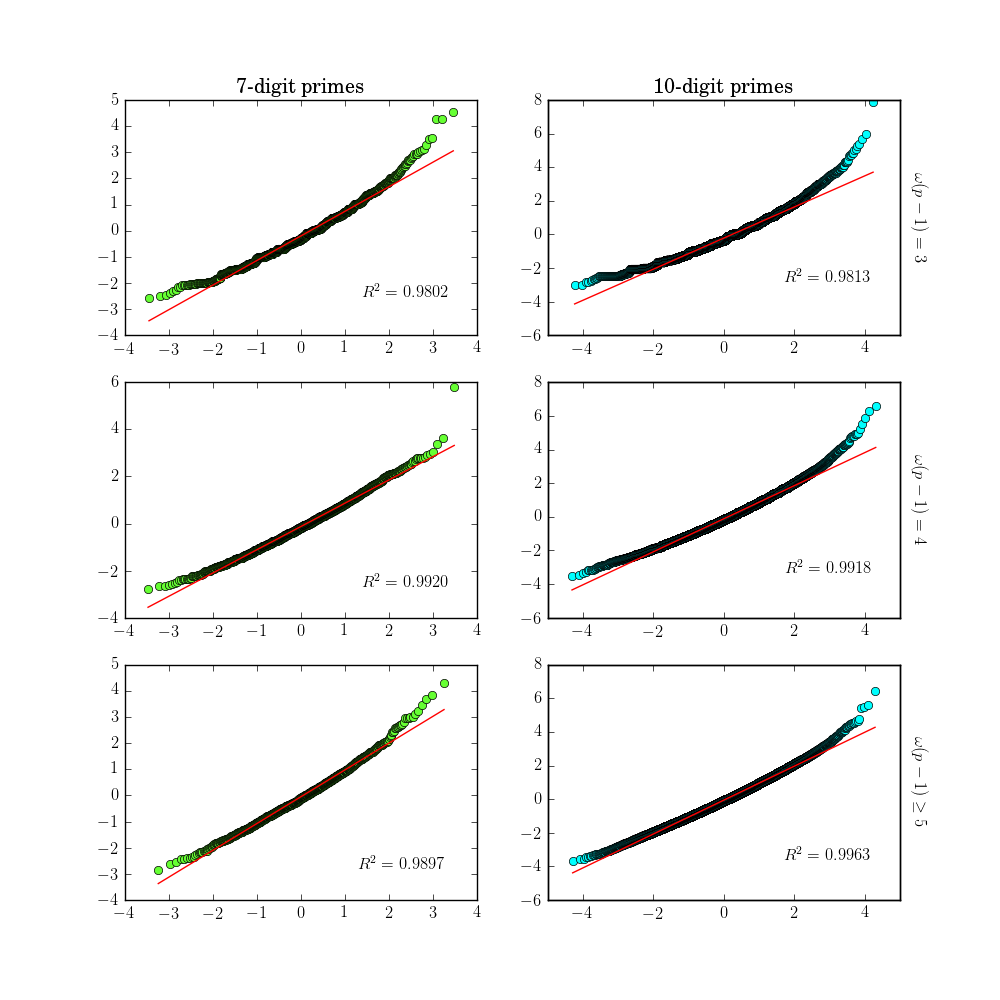}
\end{center}
\caption{Probability plots for seven- (light green) and ten-digit (light blue) primes broken up into subgroups ($\omega(p-1)=3$, $\omega(p-1)=4$ and $\omega(p-1) \ge 5$).  The group $\omega(p-1)=2$ was computed but the data gives rise to many outliers for reasons that readily ascertainable.  Note that tail divergence on both ends for these plots.}
\end{figure}

%\newpage

%\bibliographystyle{plain}
%\bibliography{fixed_points}

\begin{thebibliography}{10}

\bibitem{BBS}
Antal Balog, Kevin~A. Broughan, and Igor~E. Shparlinski.
\newblock On the number of solutions of exponential congruences.
\newblock {\em Acta Arith.}, 148(1):93--103, 2011.

\bibitem{beckenbach}
Edwin~F. Beckenbach and Richard Bellman.
\newblock {\em Inequalities}.
\newblock Second revised printing. Ergebnisse der Mathematik und ihrer
  Grenzgebiete. Neue Folge, Band 30. Springer-Verlag, New York, Inc., 1965.

\bibitem{CG2}
J.~Cilleruelo and M.~Z. Garaev.
\newblock Congruences involving product of intervals and sets with small
  multiplicative doubling modulo a prime and applications.
\newblock {\em Math. Proc. Cambridge Philos. Soc.}, 160(3):477--494, 2016.

\bibitem{CG1}
Javier Cilleruelo and Moubariz~Z. Garaev.
\newblock On the congruence $x^x\equiv \lambda \pmod p$, 2015.

\bibitem{Crocker}
Roger Crocker.
\newblock On residues of {$n^{n}$}.
\newblock {\em Amer. Math. Monthly}, 76:1028--1029, 1969.

\bibitem{FrHo}
Matthew Friedrichsen and Joshua Holden.
\newblock Statistics for fixed points of the self-power map, arxiv:1403.5548,
  2014.

\bibitem{KLS}
P{\"a}r Kurlberg, Florian Luca, and Igor~E. Shparlinski.
\newblock On the fixed points of the map {$x\mapsto x^x$} modulo a prime.
\newblock {\em Math. Res. Lett.}, 22(1):141--168, 2015.

\bibitem{loher-masser}
Thomas Loher and David Masser.
\newblock Uniformly counting points of bounded height.
\newblock {\em Acta Arith.}, 111(3):277--297, 2004.

\bibitem{MvOV}
Alfred~J. Menezes, Paul~C. van Oorschot, and Scott~A. Vanstone.
\newblock {\em Handbook of applied cryptography}.
\newblock CRC Press Series on Discrete Mathematics and its Applications. CRC
  Press, Boca Raton, FL, 1997.
\newblock With a foreword by Ronald L. Rivest.

\bibitem{pappalardi-sha-shparlinski-stewart}
Francesco Pappalardi, Min Sha, Igor~E. Shparlinski, and Cameron~L. Stewart.
\newblock On multiplicatively dependent vectors of algebraic numbers,
  arxiv:1606.02874, 2016.

\bibitem{Somer}
Lawrence Somer.
\newblock The residues of {$n^{n}$} modulo {$p$}.
\newblock {\em Fibonacci Quart.}, 19(2):110--117, 1981.

\end{thebibliography}

\end{document}